\documentclass{amsart}%
\usepackage{palatino, mathpazo}
\usepackage{amsfonts}
\usepackage{amsmath}
\usepackage{amssymb,latexsym,xcolor}
\usepackage{graphicx}
\usepackage{harpoon}
\usepackage[mathscr]{eucal}
\usepackage{amssymb}%
\usepackage[
linktocpage=true,colorlinks,citecolor=magenta,linkcolor=blue,urlcolor=magenta]{hyperref}
\providecommand{\U}[1]{\protect \rule{.1in}{.1in}}

\newtheorem{theorem}{Theorem}[section]

\newtheorem{corollary}[theorem]{Corollary}
\newtheorem{definition}[theorem]{Definition}
\newtheorem{example}[theorem]{Example}

\newtheorem{Proposition}[theorem]{Proposition}
\newtheorem{Theorem}{Theorem}

\theoremstyle{remark}
\newtheorem{Remark}[theorem]{Remark}

\numberwithin{equation}{section}

\newcommand{\p}{\partial}

\begin{document}
\title[]{Generic non-degeneracy of critical points of multiple Green functions on torus and applications to curvature equations}
\author{Zhijie Chen}
\address{Department of Mathematical Sciences, Yau Mathematical Sciences Center,
Tsinghua University, Beijing, 100084, China }
\email{zjchen2016@tsinghua.edu.cn}
\author{Erjuan Fu}
\address{Beijing Institute of Mathematical Sciences and Applications, Beijing, 101408, China}
\email{fej.2010@tsinghua.org.cn, ejfu@bimsa.cn}
\author{Chang-Shou Lin}
\address{Department of Mathematics, National Taiwan University, Taipei 10617, Taiwan }
\email{cslin@math.ntu.edu.tw}

\begin{abstract}
Let $E_{\tau}:=\mathbb{C}/(\mathbb{Z}+\mathbb{Z}\tau)$ with $\operatorname{Im}\tau>0$ be a flat torus and
$G(z;\tau)$ be the Green function on $E_{\tau}$ with the singularity at $0$. Consider the multiple Green
function $G_{n}$ on $(E_{\tau})^{n}$:
\[
G_{n}(z_{1},\cdots,z_{n};\tau):=\sum_{i<j}G(z_{i}-z_{j};\tau)-n\sum_{i=1}%
^{n}G(z_{i};\tau).
\] Recently, Lin (J. Differ. Geom. to appear) proved that there are at least countably many analytic curves in $\mathbb H=\{\tau : \operatorname{Im}\tau>0\}$ such that $G_n(\cdot;\tau)$ has degenerate critical points for any $\tau$ on the union of these curves.
In this paper, we prove that there is a measure zero subset $\mathcal{O}_n\subset \mathbb H$ (containing these curves) such that for any $\tau\in \mathbb H\setminus\mathcal{O}_n$, all critical points of $G_n(\cdot;\tau)$ are non-degenerate.
 Applications to counting the exact number of solutions of the curvature equation
$\Delta u+e^{u}=\rho \delta_{0}$ on $E_{\tau}$
 will be given.

\end{abstract}


\maketitle

\section{Introduction}

Let $\tau \in \mathbb{H}:=\{ \tau \in \mathbb C\,|\operatorname{Im}\tau>0\}$ and $E_{\tau}:=\mathbb{C}/\Lambda_{\tau}$ be a flat torus in the plane, where
$\Lambda_{\tau}=\mathbb{Z}+\mathbb{Z}\tau$. Denote also $\omega_{1}=1
$, $\omega_{2}=\tau$ and $\omega_{3}=1+\tau$. Let $G(z)=G(z;\tau)$ be
the Green function on $E_{\tau}$ defined by%
\[
-\Delta G(z;\tau)=\delta_{0}-\frac{1}{\left \vert E_{\tau}\right \vert }\text{
\ on }E_{\tau},\quad
\int_{E_{\tau}}G(z;\tau)=0,
\]
where $\delta_{0}$ is the Dirac measure at $0$ and $\vert E_{\tau}\vert$ is
the area of the torus $E_{\tau}$. It is an even function with the only
singularity at $0$. Let $E_{\tau}^{\times}:=E_{\tau}\setminus \{0\}$ and
consider the complete diagonal in $(E_{\tau}^{\times})^{n}$ for $n\in\mathbb{N}^*=\mathbb{N}\setminus\{0\}$:
\[
\Delta_{n}:=\{(z_{1},\cdot \cdot \cdot,z_{n})\in(E_{\tau}^{\times})^{n}%
|z_{j}=z_{k}\text{ for some }j\not =k\}.
\]
Let $\boldsymbol{z}=(z_{1},\cdot \cdot \cdot,z_{n})\in(E_{\tau}^{\times}%
)^{n}\backslash \Delta_{n}$ and as in \cite{CLW, CL-JDG21, Lin-JDG25, LW-CAG}, we define the \emph{multiple Green function}
$G_{n}(\boldsymbol{z})=G_{n}(\boldsymbol{z};\tau)$ on $(E_{\tau}^{\times}
)^{n}\backslash \Delta_{n}$ by%
\begin{equation}
G_{n}(\boldsymbol{z};\tau):=\sum_{j<k}G(z_{j}-z_{k};\tau)-n\sum_{j=1}%
^{n}G(z_{j};\tau).\label{multiple-G}%
\end{equation}
Note that $G_n$ is invariant under the permutation group $S_n$ and $G_1=-G$.

\medskip

\noindent{\bf Definition.} \cite{CLW}
{\it $\boldsymbol{a}=\{a_1,\cdots,a_n\}$ is called a critical point of $G_n$ if
\begin{equation}\label{poho}0=\nabla_jG_n(\boldsymbol{a}):=\sum_{k=1,\not =j}^{n}\nabla G(a_{j}-a_{k})-n\nabla G(a_{j}),\text{ \ }1\leq
j\leq n.\end{equation}
Clearly $-\boldsymbol{a}:=\{-a_1,\cdots,-a_n\}$ is also a critical point if $\boldsymbol{a}$ is. A critical point $\boldsymbol{a}$ is called trivial if
\[\{a_1,\cdots,a_n\}=\{-a_1,\cdots,-a_n\}\quad\text{in }E_{\tau};\]
 A critical point $\boldsymbol{a}$ is called nontrivial if
 \[\{a_1,\cdots,a_n\}\neq\{-a_1,\cdots,-a_n\}\quad\text{in }E_{\tau}.\]
In other words, nontrivial critical points must appear in pairs if exists.}

\medskip

Our original motivation of studying $G_{n}$ comes from its deep connection (see
\cite{CLW, CL-JDG21, Lin-JDG25, LW-CAG}) with the bubbling phenomena of semilinear elliptic partial differential equations with exponential nonlinearities in two dimension. A typical example is
 the curvature equation with parameter $\rho>0$:
\begin{equation}
\Delta u+e^{u}=\rho \delta_{0}\text{ \ on }E_{\tau}.\label{mfe}%
\end{equation}
Geometrically, a solution $u$ leads to a metric $ds^{2}=\frac{1}{2}%
e^{u}|dz|^2$ with constant Gaussian curvature $+1$ acquiring a conic
singularity at $0$. It also appears in statistical physics as the equation for
the mean field limit of the Euler flow in Onsager's vortex model (cf.
\cite{CLMP}), hence also called a mean field equation in \cite{CLW, LW, LW2} and references therein.

One important feature of (\ref{mfe}) is the so-called \emph{bubbling
phenomena}. Let $u_{k}$ be a sequence of solutions of (\ref{mfe}) with
$\rho=\rho_{k}\rightarrow8\pi n$, $n\in \mathbb{N}^*$, and $\max_{E_{\tau}}%
u_{k}(z)\rightarrow+\infty$ as $k\rightarrow+\infty$. We call $p$ a blowup point of $\{u_k\}$ if there is a sequence $\{x_k\}_k$ such that $x_k\to p$ and $u_k(x_k)\to+\infty$ as $k\to+\infty$. Then it was proved in
\cite{BT,CL-1} that $u_{k}$ always has exactly $n$ blowup points $\{a_{1}%
,\cdot \cdot \cdot,a_{n}\}$ in $E_{\tau}$ and $a_{j}\not =0$ for all $j$. Furthermore,
\[u_k(z)+\rho_kG(z)-\int_{E_{\tau}}u_k=\int_{E_\tau}G(z-y)e^{u_k(y)}dy\to 8\pi\sum_{j}G(z-a_j)\]
uniformly in $K\Subset E_{\tau}\setminus\{a_1,\cdots, a_n\}$. From here,
the
well-known Pohozaev identity says that the positions of these blowup
points could be determined by equations%
\[
n\nabla G(a_{j})=\sum_{k=1,\not =j}^{n}\nabla G(a_{j}-a_{k}),\text{ \ }1\leq
j\leq n,\]
i.e. the blowup point set $\{a_{1},\cdot \cdot \cdot,a_{n}\}$ gives a
critical point of $G_{n}$. More precisely,

\begin{Theorem}\cite{CLW}\label{Thm-A} Fix $\tau$ and let $u_{k}$ be a
sequence of bubbling solutions of \eqref{mfe} with $\rho=\rho
_{k}\rightarrow8\pi n$, $n\in \mathbb{N}^*$.
\begin{itemize}
\item[(1)] If $\rho_{k}\not =8\pi
n$ for large $k$, then the blowup set $\boldsymbol{a}=\{a_{1},\cdot \cdot \cdot
,a_{n}\}$ gives a trivial critical point of
$G_{n}$.

\item[(2)] If $\rho_{k}\equiv 8\pi
n$ for large $k$, then the blowup set $\boldsymbol{a}=\{a_{1},\cdot \cdot \cdot
,a_{n}\}$ gives a nontrivial critical point of
$G_{n}$.
\end{itemize}
Furthermore, 
\begin{equation}
\Delta u+e^{u}=8\pi n \delta_{0}\text{ \ on }E_{\tau},\label{mfe-n}%
\end{equation}
has solutions if and only if $G_{n}(\boldsymbol{z};\tau)$ has nontrivial critical points.
\end{Theorem}

Remark that in Theorem \ref{Thm-A}, if we also let $\tau=\tau_k\to \tilde{\tau}$ with $\tau_k\neq \tilde{\tau}$ for large $k$, then the assertion (1) still holds but the assertion (2) might not; see \cite{Lin-JDG25} for details.

Theorem \ref{Thm-A} highlights the importance of both trivial and nontrivial critical points of $G_n$.
Naturally we are concerned with the following fundamental questions:

\begin{itemize}
\item[(a)] \emph{How many trivial (resp. nontrivial) critical points does $G_{n}$ have? Remark that it is not obvious to see whether the set of critical points of $G_n$ is finite or not.}
\item[(b)] \emph{Is a critical point non-degenerate? Or equivalently, does the Hessian $$\det D^2G_n(\boldsymbol{a})\neq 0$$ for a critical point $\boldsymbol{a}$? Note that $D^2G_n(\boldsymbol{a})$ is a $2n\times 2n$ matrix of the second derivatives of $G_n$, so $\det D^2G_n(\boldsymbol{a})$ is difficult to compute in general}.
\end{itemize}

The case $n=1$ (note $G_{1}=-G$) was first studied by Lin and Wang \cite{LW}. Since $G(z)$ is even and doubly periodic, it
always has three trivial critical points $\frac{\omega_{k}}{2}$, $k=1,2,3$,
and nontrivial critical points must appear in pairs if exist. It was proved in \cite{LW} that
$G(z;\tau)$ has \emph{at most one pair of nontrivial critical
points} (depending on the choice of $\tau$). For example, if $\tau=e^{\pi i/3}$, then $G(z;\tau)$ has exactly one pair of nontrivial critical points $\pm\frac{\omega_3}{3}$; while for $\tau \in i\mathbb{R}_{>0}$, i.e. $E_{\tau}$ is a rectangular torus, $G(z;\tau)$ has no nontrivial critical points.
Later, a complete characterization of those $\tau$'s such that $G(z;\tau)$
has a pair of nontrivial critical points was given in \cite{BE,CKLW,LW4}.

For general $n\geq2$, $G_n$ was first studied by Chai, Lin and Wang \cite{CLW} from the viewpoint of algebraic geometry. Among other
things, they established the one-to-one correspondence between
trivial critical points of $G_{n}$ and branch points of a hyperelliptic curve
$Y_{n}$ concerning the classical Lam\'{e} equation; see Section 2 for a brief review. Consequently, they can solve Question (a) for trivial critical points.

\begin{Theorem}\cite{CLW}\label{Thm-B} $G_{n}(\boldsymbol{z};\tau)$ always has trivial critical points, and the number of trivial critical points is
 at most $2n+1$ (resp. exactly $2n+1$) for any
$\tau$ (resp. for almost all $\tau$ including $\tau \in i\mathbb{R}_{>0}$).
\end{Theorem}

Recently, Chen and Lin \cite{CL-JDG21} obtained a complete answer of Questions (a)-(b) for $\tau \in i\mathbb{R}_{>0}$, i.e. $E_{\tau}$ is a rectangular torus.

\begin{Theorem}\cite{CL-JDG21}
\label{Thm-C} Let $\tau \in i\mathbb{R}_{>0}$. Then $G_{n}%
(\boldsymbol{z};\tau)$ has exactly
$2n+1$ critical points, which are all trivial critical points. Furthermore, they are all non-degenerate, i.e.
 $\det D^{2}G_{n}(\boldsymbol{a};\tau)$ $\neq0$ for any critical point
$\boldsymbol{a}$.
\end{Theorem}

However, Questions (a)-(b) remain largely open for general $\tau\in\mathbb H$. By Theorem \ref{Thm-A}, we define
\begin{align}\label{eq-en}\mathcal{E}_n:=&\{\tau\in\mathbb H\;:\; G_n(\boldsymbol{z};\tau)\;\text{has nontrivial critical points}\}\\
=& \{\tau\in\mathbb H\;:\; \text{The curvature equation \eqref{mfe-n} has solutions}\}.\nonumber
\end{align}
Then Theorem \ref{Thm-C} implies $i\mathbb{R}_{>0}\cap\mathcal{E}_n=\emptyset$.
Recently, Lin \cite{Lin-JDG25} gave a characterization of $\mathcal{E}_n$ and, among other things, proved that $\mathcal{E}_n\neq\emptyset$ is an open set.

\begin{Theorem}\cite{Lin-JDG25}
$\mathcal{E}_n\neq\emptyset$ is an unbounded and open subset in $\mathbb{H}$, and $\overline{\mathcal{E}_n}\subsetneqq \mathbb H\setminus i\mathbb{R}_{>0}$, where $\overline{\mathcal{E}_n}$ is the closure of $\mathcal{E}_n$ in $\mathbb H$.
\end{Theorem}

The main result of this paper is to answer Question (b) for generic $\tau\in\mathbb H$.

\begin{theorem}\label{main-thm-1} Define
\[\mathcal O_n:=\{\tau\in \mathbb H\;:\; \text{$G_n(\boldsymbol{z};\tau)$ has degenerate critical points}\}\subset \mathbb H\setminus i\mathbb {R}_{>0}.\]
Then $\mathcal O_n$ is of Lebesgue measure zero. In other words, for almost all $\tau$, any critical point of $G_n(\boldsymbol{z};\tau)$ is non-degenerate.
\end{theorem}

Theorem \ref{main-thm-1} is sharp in the sense that it was already proved in \cite{Lin-JDG25} that $\mathcal O_n$ contains at least countably many analytic curves. These curves were called Lin-Wang curves by \cite{Eremenko-GT}, along which $G_{n}(\boldsymbol z; \tau)$ has degenerate trivial critical points. This result will be recalled in the proof of Theorem \ref{main-thm-1} in Section 4.  The novelty of this paper is to prove that the set of those $\tau$'s, for which $G_{n}(\boldsymbol z; \tau)$ has degenerate nontrivial critical points, is of Lebesgue measure zero.

 It is well known that the non-degeneracy of critical points has important applications in constucting solutions via the reduction approach
and proving the local uniqueness of bubbling solutions for elliptic PDEs; see e.g. \cite{BKLY,CL-JDG21,LY} and references therein for many examples, where the non-degeneracy of critical points often appreared as assumptions.  We believe that Theorem \ref{main-thm-1} has important applications in future.
Here Theorem \ref{main-thm-1} can be applied to count the exact number of solutions for the curvature equation \eqref{mfe}, as stated in the next result.
\begin{theorem}
\label{numbers1} Let $n\in \mathbb{N}^*$ and recall \eqref{eq-en} that the curvature equation \eqref{mfe} with $\rho=8\pi n$ has no solutions for each $\tau\in \mathbb{H}\setminus \mathcal{E}_n$. Then for each $\tau\in \mathbb{H}\setminus (\mathcal{E}_n\cup\mathcal{O}_n)$, there exists 
small $\varepsilon_{\tau}>0$ such that

\begin{itemize}
\item[(1)] for $\rho \in(8\pi n, 8\pi n+\varepsilon_{\tau})$, \eqref{mfe} has
exactly $n+1$ solutions;

\item[(2)] for $\rho \in(8\pi n-\varepsilon_{\tau}, 8\pi n)$, \eqref{mfe} has
exactly $n$ solutions.
\end{itemize}
\end{theorem}

\begin{proof}
Note that $i\mathbb{R}_{>0}\subsetneqq \mathbb{H}\setminus (\mathcal{E}_n\cup\mathcal{O}_n)$. 
Theorem \ref{numbers1} is a generalization of \cite[Theorem 1.3]{CL-JDG21}, where the same statements (1)-(2) were proved for $\tau\in i\mathbb{R}_{>0}$ because of the validity of Theorem \ref{Thm-C}. 

Now for each $\tau\in \mathbb{H}\setminus (\mathcal{E}_n\cup\mathcal{O}_n)$, we have that:
\begin{itemize}
\item[(a)] The definitions of $\mathcal{E}_n$ and $\mathcal{O}_n$ imply that all critical points of $G_n(\boldsymbol{z};\tau)$ are trivial and non-degenerate.
\item[(b)] Suppose the number of trivial critical points of $G_n(\boldsymbol z;\tau)$ is less than $2n+1$. Since Theorem \ref{Thm-B} implies the existence of $\tau_k\to \tau$ such that $G_n(\boldsymbol z;\tau_k)$ has exactly $2n+1$ trivial critical points, it follows that some trivial critical point of $G_n(\boldsymbol z;\tau)$ must be degenerate, i.e. $\tau\in\mathcal{O}_n$, a contradiction. Therefore, $G_n(\boldsymbol z;\tau)$ has exactly $2n+1$ trivial critical points.
\end{itemize}
Therefore, the statements of Theorem \ref{Thm-C} also hold for $\tau\in \mathbb{H}\setminus (\mathcal{E}_n\cup\mathcal{O}_n)$. Consequently, the same proof as \cite[Theorem 1.3]{CL-JDG21} implies Theorem \ref{numbers1}.
\end{proof}

In the proof of Theorem \ref{main-thm-1}, our new idea is to prove a precise formula of
 the Hessian $\det D^2G_n(\boldsymbol{a})$ for a nontrivial critical point $\boldsymbol{a}$, which gives a criterion for the degeneracy of nontrivial critical points.

\begin{theorem} \label{DD-conj0}Fix any $n \in \Bbb N^*$ and $\tau$. Suppose $\boldsymbol{p} = \{p_1, \cdots, p_n\}$ is a nontrivial critical point of $G_n$, then there is a constant $c_{\boldsymbol{p}} >0 $ such that
\begin{equation}\label{eq-ex-non0}
\det D^2 G_n(\boldsymbol{p};\tau) = \frac{(-1)^n n^2}{4(2\pi)^{2n+2} \operatorname{Im}\tau} c_{\boldsymbol{p}} |\tau_r(r_{\boldsymbol{p}}, s_{\boldsymbol{p}})|^2\operatorname{Im}
\left(\frac{\tau_s}{\tau_r}(r_{\boldsymbol{p}}, s_{\boldsymbol{p}})\right).
\end{equation}
Consequently, $\boldsymbol{p}$ is degenerate if and only if
$|\tau_r(r_{\boldsymbol{p}}, s_{\boldsymbol{p}})|^2\operatorname{Im}
(\frac{\tau_s}{\tau_r}(r_{\boldsymbol{p}}, s_{\boldsymbol{p}}))=0$, namely
\[\frac{\tau_s}{\tau_r}(r_{\boldsymbol{p}}, s_{\boldsymbol{p}})\in\mathbb{R}\cup\{\infty\}.\]
Here $(r_{\boldsymbol{p}}, s_{\boldsymbol{p}})\in\mathbb{R}^2\setminus\frac12\mathbb{Z}^2$ is determined by $\sum_{j=1}^np_j=r_{\boldsymbol{p}}+s_{\boldsymbol{p}}\tau$, and $\tau(r,s)$ is some holomorphic function satisfying $\tau(r_{\boldsymbol{p}}, s_{\boldsymbol{p}})=\tau$ that will be introduced in Remark \ref{Rmk2-4}, with $\tau_r:=\frac{\partial \tau}{\partial r}$, $\tau_s:=\frac{\partial \tau}{\partial s}$ and $(\tau_r, \tau_s)\neq (0,0)$. 
\end{theorem}

Remark that the $(r_{\boldsymbol{p}}, s_{\boldsymbol{p}})$ also appears as the monodromy data of the classical Lam\'{e} equation (a second order Fuchsian ODE); see Section 2 for a brief overview.

The rest of this paper is organized as follows. In Section 2, we will introduce the deep connection between critical points of $G_n$ and the monodromy theory of the Lam\'{e} equation. In Section 3, we will prove Theorem \ref{DD-conj0}. In Section 4, we will prove Theorem \ref{main-thm-1}.

\section{Connection with the Lam\'{e} equation}

In this section, we
briefly review the connection between critical points of $G_{n}$ and
the classical Lam\'{e} equation from \cite{CLW}.

Let $\wp(z)=\wp( z;\tau)$ be the Weierstrass elliptic function with periods
$\Lambda_{\tau}$, defined by%
\begin{equation}
\label{40-3}\wp(z;\tau):=\frac{1}{z^{2}}+\sum_{\omega \in \Lambda_{\tau
}\backslash\{0\}  }\left(  \frac{1}{(z-\omega)^{2}}-\frac
{1}{\omega^{2}}\right) ,
\end{equation}
which satisfies the well-known cubic equation
\begin{equation}
\label{40-4}\wp^{\prime}(z;\tau)^{2}=4\wp(z;\tau)^{3}-g_{2}(\tau)\wp
(z;\tau)-g_{3}(\tau),
\end{equation}
where $g_2, g_3$ are known as invariants of the elliptic curve $E_{\tau}$.

Let $\zeta(z)=\zeta(z;\tau):=-\int^{z}\wp(\xi;\tau)d\xi$ be the Weierstrass
zeta function with two quasi-periods
\begin{equation}
\eta_j=\eta_{j}(\tau):=2\zeta(\tfrac{\omega_{j}}{2};\tau)=\zeta(z+\omega_{j}%
;\tau)-\zeta(z;\tau),\quad j=1,2,\label{40-2}%
\end{equation}
and $\sigma(z)=\sigma(z;\tau):=\exp(\int^{z}\zeta(\xi;\tau)d\xi)$ be the Weierstrass sigma function. Notice that $\zeta(z)$ is an odd meromorphic function with simple poles at $\Lambda_{\tau}$, and $\sigma(z)$ is an odd
entire function with simple zeros at $\Lambda_{\tau}$, which satisfies the following transformation law
\begin{equation}\label{40-3}
\sigma(z+\omega_{j})=-e^{\eta_{j}(z+\omega_{j}/2)}\sigma(z),\text{ \ }j=1,2.
\end{equation}

\subsection{The Lam\'{e} equation}

Consider the well-known integral Lam\'{e} equation
\begin{equation}\tag{$\mathcal{L}_{n,B,\tau}$}
y^{\prime \prime}(z)=[n(n+1)\wp
(z;\tau)+B]y(z),\quad z\in\mathbb{C},
\end{equation}
where $n\in \mathbb{N}^*$ and $B\in \mathbb{C}$ are called index
and accessory parameter respectively. 
See e.g. the classic text \cite{Whittaker-Watson} and recent works \cite{CLW,Dahmen0,Dahmen,LW2,Maier} for introductions about $\mathcal{L}_{n,B,\tau}$.
In particular, thanks to $n\in\mathbb{N}^*$, it is easy to see (cf. \cite{CLW,Dahmen0,Whittaker-Watson}) that any solution of $\mathcal{L}_{n,B,\tau}$ is single-valued and meromorphic in $\mathbb{C}$.

For $\boldsymbol{a}=\{a_{1},\cdot \cdot \cdot,a_{n}\}\in E_{\tau}^n/S_n$, where $S_n$ denotes the permutation group, we consider the
\textit{Hermite-Halphen ansatz} (cf. \cite{CLW,Whittaker-Watson}):
\begin{equation}
y_{\boldsymbol{a}}(z):=e^{z\sum_{j=1}^{n}\zeta(a_{j})}\frac{\prod_{j=1}%
^{n}\sigma(z-a_{j})}{\sigma(z)^{n}}.\label{ff2-1}%
\end{equation}
Following also \cite{CLW}, we define
\begin{equation}
\label{hyper}Y_{n}=Y_n(\tau):=\left \{  \boldsymbol{a}\in E_{\tau}^n/S_n\left \vert
\begin{array}
[c]{l}%
a_{j}\neq0,\text{ }a_{j}\neq a_{k}\text{ in $E_{\tau}$ for any }j\not =k,\\
\sum_{k\not =j}^{n}(\zeta(a_{j}-a_{k})+\zeta(a_{k})-\zeta(a_{j}))=0,\text{
}\forall j
\end{array}
\right.  \right \}  .
\end{equation}
Clearly $-\boldsymbol{a}:=\{-a_{1},\cdot \cdot \cdot,-a_{n}\} \in Y_{n}$ if
$\boldsymbol{a}\in Y_{n}$. The connection between the Lam\'{e} equation and $Y_n$ 
has been studied in the literature.

\begin{theorem}
(cf. \cite{CLW,Whittaker-Watson}) \label{CLWWW}
\begin{itemize}
\item[(1)]
For any $B\in\mathbb C$, there exists a unique pair $\pm\boldsymbol{a}\in Y_n(\tau)$ such that $y_{\boldsymbol{a}}(z)$ and $y_{-\boldsymbol{a}}(z)$ are both solutions of the Lam\'{e} equation $\mathcal{L}_{n,B,\tau}$, and
\begin{equation}
\label{b-a}B=B_{\boldsymbol{a}} :=(2n-1)\sum_{j=1}^{n}\wp(a_{j};\tau).
\end{equation}
\item[(2)] There is a so-called Lam\'e polynomial $\ell_{n}(B)=\ell_{n}(B;\tau
)\in \mathbb{Q}[g_{2}(\tau),g_{3}(\tau)][B]$ of degree $2n+1$ such that $y_{\boldsymbol{a}}(z)$ and $y_{-\boldsymbol{a}}(z)$ are linearly independent if and only if $\ell_n(B;\tau)\neq 0$. In this case, $\boldsymbol{a}\neq -\boldsymbol{a}$ and indeed,
\begin{equation}
\{a_{1},\cdot \cdot \cdot,a_{n}\} \cap \{-a_{1},\cdot \cdot \cdot,-a_{n}%
\}=\emptyset\quad\text{in }\;E_{\tau}.\label{fc6}%
\end{equation}
Besides, we define $(r,s)\in\mathbb{C}^2$ by
\begin{equation}\label{de-rs}\left \{
\begin{array}
[c]{l}%
r+s\tau=\sum_{j=1}^{n}a_j\\
r\eta_{1}+s\eta_{2}=\sum_{j=1}^{n}\zeta(a_{j}),
\end{array}
\right.\end{equation}
which can be uniquely solved by the Legendre relation $\tau\eta_1-\eta_2=2\pi i$.
Then $(r,s)\notin\frac12\mathbb{Z}^2:=\{(r,s) \;:\; 2r,2s\in\mathbb Z\}$ and
\begin{equation}\label{mono1}\begin{pmatrix}y_{\boldsymbol{a}}(z+\omega_1)\\ y_{-\boldsymbol{a}}(z+\omega_1)\end{pmatrix}=\begin{pmatrix}
e^{-2\pi is} & 0\\
0 & e^{2\pi is}%
\end{pmatrix}\begin{pmatrix}y_{\boldsymbol{a}}(z)\\ y_{-\boldsymbol{a}}(z)\end{pmatrix},\end{equation}
\begin{equation}\label{mono2}\begin{pmatrix}y_{\boldsymbol{a}}(z+\omega_2)\\ y_{-\boldsymbol{a}}(z+\omega_2)\end{pmatrix}=\begin{pmatrix}
e^{2\pi ir} & 0\\
0 & e^{-2\pi ir}%
\end{pmatrix}\begin{pmatrix}y_{\boldsymbol{a}}(z)\\ y_{-\boldsymbol{a}}(z)\end{pmatrix}.\end{equation}
\item[(3)] The
map $B:Y_{n}(\tau)\rightarrow \mathbb{C}$ defined by \eqref{b-a} is a ramified
covering of degree $2$, and 
\[
Y_{n}(\tau)\cong \left \{  (B,C)\text{ }|\text{ }C^{2}=\ell_{n}(B;\tau)\right \}
\]
 is a
hyperelliptic curve with $\boldsymbol{a}\in Y_{n}(\tau)$ being a branch
point if $\boldsymbol{a}=-\boldsymbol{a}$ in $E_{\tau}$.
\end{itemize}
\end{theorem}



Remark that \eqref{mono1}-\eqref{mono2} follow directly from \eqref{40-2}, \eqref{40-3}, \eqref{ff2-1} and \eqref{de-rs}, and $-\boldsymbol{a}$ corresponds to $(-r, -s)$. Note also that if $\tilde{a}_j\in a_j+\Lambda_{\tau}$, then $\tilde{\boldsymbol{a}}:=\{\tilde{a}_1,\cdots,\tilde{a}_n\}=\boldsymbol{a}$ in $E_{\tau}$, $y_{\pm\tilde{\boldsymbol{a}}}(z)=C_{\pm}y_{\pm\boldsymbol a}(z)$ for some constants $C_{\pm}\neq 0$, and the corresponding $(\tilde{r},\tilde{s})$ determined by $\tilde{\boldsymbol a}$ via \eqref{de-rs} satisfies $(\tilde{r},\tilde{s})\equiv (r,s)$ mod $\mathbb{Z}^2$. Thus, $(r,s)\in(\mathbb{C}^{2}\backslash \frac{1}{2}\mathbb{Z}^{2})/\sim$ is uniquely determined by $(\tau,B)$ and is known as the monodromy data of $\mathcal{L}_{n,B,\tau}$ (cf. \cite{CL-AIM25, LW2}), because the monodromy group of $\mathcal{L}_{n,B,\tau}$ is precisely generated by the two matrices in \eqref{mono1}-\eqref{mono2}. Here $\sim$ is an equivalent relation defined by\[(r,s)\sim (\tilde r,\tilde s)\quad\text{ if}\quad (r,s)\equiv \pm (\tilde r,\tilde s)\quad\operatorname{mod} \mathbb{Z}^2. \]
Define
\begin{equation}
\Sigma_{n}:=\{(\tau, B)\in \mathbb{H}\times \mathbb{C}\,|\, \ell_n(B;\tau)\neq 0\},
\end{equation}
which is clearly an open connected subset of $\mathbb{H}\times \mathbb{C}$.
Then by the above argument, the map $\Sigma_n\ni (\tau, B)\mapsto (r,s)\in (\mathbb{C}^{2}\backslash \frac{1}{2}\mathbb{Z}^{2})/\sim$ is well-defined. 
\begin{Remark}\label{Rmk2-2}
Given any $(\tau_0, B_0)\in \Sigma_{n}$, take $(r_0,s_0)\in \mathbb{C}^{2}\backslash \frac{1}{2}\mathbb{Z}^{2}$ to be a representative of  the monodromy data of $\mathcal{L}_{n,B_0,\tau_0}$. Since there is a small neighborhood $V\subset \mathbb{C}^{2}\backslash \frac{1}{2}\mathbb{Z}^{2}$ of $(r_0,s_0)$ such that $(r,s)\not\sim (r',s')$ for any $(r,s), (r',s')\in V$ satisfying $(r,s)\neq (r',s')$, there is a small neighborhood $U\subset \Sigma_{n}$ of $(\tau_0, B_0)$ such that $U\ni(\tau,B)\mapsto (r,s)\in(\mathbb{C}^2\setminus\tfrac12\mathbb{Z}^2)/\sim$ can be seen as 
\begin{equation}\label{eq-UV} U\ni(\tau,B)\mapsto (r,s)\in V\subset \mathbb{C}^2\setminus\tfrac12\mathbb{Z}^2,\end{equation}
and so we can consider the local analytic property of this map. Recently, Chen and Lin \cite{CL-AIM25} proved a surprising universal law for this map.
\end{Remark}

\begin{theorem}\cite{CL-AIM25}\label{thm-rsbtau} The map  $\Sigma_n\ni (\tau, B)\mapsto (r,s)\in (\mathbb{C}^{2}\backslash \frac{1}{2}\mathbb{Z}^{2})/\sim$ is holomorphic and locally one-to-one, and satisfies
\begin{equation}\label{brstau}d\tau\wedge dB\equiv 8\pi^2 dr\wedge ds,\quad\forall (\tau, B)\in \Sigma_{n}.\end{equation}
\end{theorem}

\begin{Remark}\label{Rmk2-4}
Theorem \ref{thm-rsbtau} will play a key role in our proof of Theorem \ref{DD-conj0}. Indeed, by Theorem \ref{thm-rsbtau} and by taking the open neighborhoods $U, V$ smaller if necessary, it follows from the inverse mapping theorem that the inverse map of the map given in \eqref{eq-UV}
\begin{equation}\label{eq-UV-1} \mathbb{C}^2\setminus\tfrac12\mathbb{Z}^2\supset V\ni (r,s)\mapsto (\tau, B)\in U\subset\Sigma_n\end{equation}
is well-defined and also holomorphic. Denote $\tau_r=\frac{\partial \tau}{\partial r}$, $r_{\tau}=\frac{\partial r}{\partial\tau}$ and so on. Then
\[\begin{pmatrix}\tau_r & \tau_s\\ B_r& B_s\end{pmatrix}\begin{pmatrix}r_{\tau} & r_B\\ s_{\tau}& s_B\end{pmatrix}=\begin{pmatrix}1 & 0\\ 0& 1\end{pmatrix}.\]
Since \eqref{brstau} implies $\tau_rB_s-\tau_sB_r\equiv 8\pi^2$, we obtain that $(\tau_r, \tau_s)\neq (0,0)$ and
\[\begin{pmatrix}B_s & -\tau_s\\ -B_r& \tau_r\end{pmatrix}={8\pi^2}\begin{pmatrix}r_{\tau} & r_B\\ s_{\tau}& s_B\end{pmatrix},\]
namely
\begin{equation}\label{eq2-16}
\tau_s=-8\pi^2r_B,\qquad \tau_r=8\pi^2s_B,\quad (\tau_r, \tau_s)\neq (0,0).
\end{equation}
This \eqref{eq2-16} will be applied in the proof of Theorem \ref{DD-conj0}.
\end{Remark}

For later usage, we also need the so-called pre-modular form $Z^{\mathbf{n}}_{r,s}(\tau)$ introduced by \cite{LW2}, 
which characterizes the monodromy data $(r,s)$ in a precise way.

\begin{definition}\cite{LW2}
A function $f_{r,s}(\tau)$ on $\mathbb{H}$, which depends meromorphically on two parameters $(r,s) (\operatorname{mod} \mathbb{Z}^2)\in\mathbb{C}^2$, is called a pre-modular form of weight $k$ if the following hold:
\begin{itemize}
\item[(1)] If $(r,s)\in\mathbb{C}^2\setminus\frac{1}{2}\mathbb{Z}^2$, then $f_{r,s}(\tau)\not\equiv 0,\infty$ and is meromorphic in $\tau$. Furthermore, it is holomorphic in $\tau$ if $(r,s)\in\mathbb{R}^2\setminus\frac{1}{2}\mathbb{Z}^2$.
\item[(2)] There is $k\in\mathbb{N}^*$ independent of $(r,s)$ such that if $(r,s)$ is any $m$-torsion point for some $m\geq 3$, then $f_{r,s}(\tau)$ is a modular form of weight $k$ with respect to $$\Gamma(m):=\{\gamma
\in SL(2,\mathbb{Z})|\gamma\equiv I_2\mod m \}.$$
\end{itemize}
\end{definition}
For example, we define
\begin{align}\label{z-rs}Z_{r,s}(\tau):=\zeta(r+s\tau;\tau)-r\eta_1(\tau)-s\eta_2(\tau).\end{align}
This $Z_{r,s}(\tau)$ was first introduced by Hecke \cite{Hecke}, and is a pre-modular form of weight one.
The main result of \cite{LW2} is following

\begin{theorem} \cite{LW2}
\label{thm-premodular}  There exists a pre-modular form $$Z_{r,s}^{(n)
}(\tau)\in\mathbb Q[Z_{r,s}(\tau),\wp(r+s\tau;\tau),\wp'(r+s\tau;\tau),g_2(\tau),g_3(\tau)]$$ of weight $\frac{1}{2}n(n+1)$ such that the following holds: 
 Given $(r,s)\in\mathbb{C}^2\setminus\frac{1}{2}\mathbb{Z}^2$ and $\tau_0\in\mathbb{H}$, there is $B\in\mathbb{C}$ such that $(r,s)$ is the monodromy data of $\mathcal{L}_{n,B,\tau_0}$, i.e. \eqref{mono1}-\eqref{mono2} holds for $\mathcal{L}_{n,B,\tau_0}$,
if and only if $Z_{r,s}^{(n)}(\tau_0)=0$.
\end{theorem}
For example, write $Z=Z_{r,s}(\tau)$, $\wp=\wp(r+s\tau;\tau)$ and $\wp^{\prime
}=\wp^{\prime}(r+s\tau;\tau)$ for convenience, then
\[
Z_{r,s}^{(1)}(\tau)=Z,\quad\quad Z_{r,s}^{(2)}(\tau)=Z^{3}-3\wp Z-\wp^{\prime},
\]
\begin{align*}
Z_{r,s}^{(3)}(\tau)=  &  Z^{6}-15\wp Z^{4}-20\wp^{\prime}Z^{3}+\left(
\tfrac{27}{4}g_{2}-45\wp^{2}\right)  Z^{2}\\
&  -12\wp \wp^{\prime}Z-\tfrac{5}{4}(\wp^{\prime})^{2}.
\end{align*}
A direct consequence of Remark \ref{Rmk2-4} and Theorem \ref{thm-premodular} is

\begin{corollary}\label{Coro-11}
Recalling the local map $V\ni (r,s)\mapsto (\tau, B)\in U$ in \eqref{eq-UV-1}, there holds
\[Z_{r,s}^{(n)}(\tau(r,s))\equiv 0.\]
\end{corollary}

Further interesting properties of $Z_{r,s}^{(n)
}(\tau)$ can be found in \cite{CKL-Dahmen,Lin-JDG25}. For example,

\begin{theorem}\cite{CKL-Dahmen}\label{thm-sz} For any fixed $n$ and $(r,s)\in\mathbb{C}^2\setminus\frac{1}{2}\mathbb{Z}^2$, $Z_{r,s}^{(n)}(\cdot)$ has at most simple zeros in $\mathbb{H}$.
\end{theorem}

A different but simpler proof of Theorem \ref{thm-sz} was recently given by \cite{CL-AIM25}.

\subsection{Connection between Lam\'{e} equation and multiple Green function}
The Green function $G$ on $E_{\tau}$ can be expressed explicitly in terms of elliptic
functions. In particular, it was proved in \cite{LW} that
\begin{equation}
\label{G_z}-4\pi G_z(z;\tau)=-4\pi \frac{\partial G}{\partial z}(z;\tau)=\zeta(z)-r\eta_{1}%
-s\eta_{2}=\zeta(z)-z\eta_{1}+2\pi i s,
\end{equation}
where the Legendre relation $\tau\eta_1-\eta_2=2\pi i$ is used and $(r,s)$ is defined by  \begin{equation}\label{realrs}z=r+s\tau\quad\text{ with }\quad r,s\in \mathbb{R}.\end{equation} By (\ref{G_z})-\eqref{realrs}, the critical point
equations (\ref{poho}) can be translated into the following equivalent
system:
\begin{equation}
\sum_{k\not =j}(\zeta(a_{j}-a_{k})+\zeta(a_{k})-\zeta(a_{j}))=0,\text{
}\forall1\leq j\leq n\label{ff3-1}%
\end{equation}
(there are only $n-1$ independent equations), and
\begin{equation}
\sum_{j=1}^{n}\nabla G(a_{j})=0,\label{ff4}%
\end{equation}
subject to the constraint $\boldsymbol{a}\in(E_{\tau}^{\times})^{n}%
\backslash \Delta_{n}$, i.e.
\begin{equation}
a_{j}\not =0,\quad a_{j}\not =a_{k}\quad \text{in}\;E_{\tau},\quad \forall
j\not =k.\label{fc5}%
\end{equation}
See \cite{CLW} for the proof. Recalling the definition \eqref{hyper} of $Y_n$, we conclude that $\boldsymbol{a}$ is a critical point of 
$G_n$ if and only if $\boldsymbol a\in Y_n$ and satisfies \eqref{ff4}. On the other hand, since $G(z)=G(-z)$ is even, \eqref{ff4} holds automatically for any branch point of $Y_n$. Therefore the above arguments yield

\begin{theorem}\cite{CLW}\label{thm-CLW}
\begin{itemize}
\item[(1)] $\boldsymbol a$ is a trivial critical point of $G_n$ if and only if $\boldsymbol a$ is a branch point of $Y_n$ (i.e. $\ell_n(B_{\boldsymbol a})=0$). Consequently, there is an one-to-one correspondence between trivial critical points of $G_n$ and distinct roots of $\ell_n$.
\item[(2)] $\boldsymbol{a}$ is a nontrivial critical point of $G_n$ if and only if $\boldsymbol{a}\in Y_n$ is not a branch point (i.e. $\ell_n(B_{\boldsymbol{a}})\neq 0$) and satisfies \eqref{ff4}. In particular, for a nontrivial critical point $\boldsymbol{a}$ of $G_n$, we have
\[\{a_1,\cdots, a_n\}\cap\{-a_1,\cdots, -a_n\}=\emptyset \quad\text{in }E_{\tau},\]
and so
\begin{equation}\label{eq-ex-aj}2a_j\notin \Lambda_{\tau}=\mathbb{Z}+\mathbb{Z}\tau, \;\text{ i.e. }\;\wp'(a_j)\neq 0, \infty, \quad \forall j.\end{equation}
\end{itemize}
\end{theorem}

Note that Theorem \ref{Thm-B} is a direct consequence of Theorem \ref{thm-CLW}-(1), because it is well known (cf. \cite{CLW}) that the Lam\'{e} polynomial $\ell_n(\cdot;\tau)$ has $2n+1$ distinct roots for almost all $\tau$'s (including $\tau\in i\mathbb{R}_{>0}$).

Suppose $\boldsymbol{a}=\{a_1,\cdots,a_n\}$ is a nontrivial critical point of $G_n(\cdot;\tau)$.
Notice that if we write
$$\sum_{j=1}^na_j=r+s\tau,\quad\text{ with }\quad r,s\in\mathbb{R},$$
then it follows from (\ref{G_z})-\eqref{realrs} that \eqref{ff4} is equivalent to 
\[\sum_{j=1}^n\zeta(a_j)=r\eta_1+s\eta_2,\]
which is precisely the second equation of \eqref{de-rs}. Then it follows from Theorem \ref{CLWWW} that $(r,s)\notin\frac{1}{2}\mathbb{Z}^2$ is the monodromy data of $\mathcal{L}_{n,B_{\boldsymbol a},\tau}$, so Theorem \ref{thm-premodular} implies $Z_{r,s}^{(n)}(\tau)=0$. The converse statement also holds, which has been proved by Lin and Wang \cite{LW2}.

\begin{theorem}\cite{LW2}\label{thm-clw2}
\begin{itemize}
\item[(1)] Suppose $\boldsymbol{a}=\{a_1,\cdots,a_n\}$ is a nontrivial critical point of $G_n(\cdot;\tau)$, and write
\begin{equation}\label{sumaj} \sum_{j=1}^na_j=r+s\tau,\quad\text{ with }\quad r,s\in\mathbb{R}.
\end{equation}
Then $(r,s)\in\mathbb{R}^2\setminus\frac{1}{2}\mathbb{Z}^2$ and $Z_{r,s}^{(n)}(\tau)=0$.
\item[(2)]  Let $(r,s)\in\mathbb{R}^2\setminus\frac{1}{2}\mathbb{Z}^2$ such that $Z_{r,s}^{(n)}(\tau)=0$. Then $G_n(\cdot;\tau)$ has a unique nontrivial critical point $\boldsymbol{a}=\{a_1,\cdots,a_n\}$ satisfying $\sum_{j=1}^na_j=r+s\tau$.
\end{itemize} 
\end{theorem}



\section{A formula for the Hessian of nontrivial critical points}

This section is denoted to the proof of Theorem \ref{DD-conj0}. Fix any $\tau$ in this section.
Define
\[\Sigma_{n}(\tau):=\{B\in\mathbb{C} :  (\tau, B)\in\Sigma_n\}=\{B\in\mathbb{C} :  \ell_n(B;\tau)\neq 0\}.\]
Recall the hyperelliptic curve $Y_n(\tau)$ defined in \eqref{hyper}.
Then it follows from Theorem \ref{CLWWW} that $B\in \Sigma_{n}(\tau)$ can be a local coordinate for the hyperelliptic curve $Y_{\mathbf{n}}(\tau)$, namely for any $j$, $a_j=a_j(B)$ is locally holomorphic for $B\in \Sigma_{n}(\tau)$. From here and \eqref{de-rs}, i.e. \[\left \{
\begin{array}
[c]{l}%
r+s\tau=\sum_{j=1}^{n}a_j(B)\\
r\eta_{1}+s\eta_{2}=\sum_{j=1}^{n}\zeta(a_{j}(B)),
\end{array}
\right.\] we conclude that the corresponding $(r,s)=(r(B), s(B))$ is locally holomorphic for $B\in \Sigma_{n}(\tau)$. Of course, this fact can also follow directly from Theorem \ref{thm-rsbtau} and $(r(B), s(B))=(r(\tau,B), s(\tau, B))$ with fixed $\tau$, where $(r(\tau,B), s(\tau, B))$ is the map in Theorem \ref{thm-rsbtau}. Consequently,
\begin{equation}
\left \{
\begin{array}
[c]{l}%
d_0:=\sum_{j=1}^{n}a_{j}'(B)=r_B+s_B\tau \\
c_0:=-\sum_{j=1}^{n}\wp(a_{j})a_j'(B)=r_B\eta_{1}+s_B\eta_{2},
\end{array}
\right. \label{r1s2}%
\end{equation}
which is equivalent to (note $\tau\eta_1-\eta_2=2\pi i$)
\begin{align*}
s_B=\frac{c_0-d_0\eta_1}{-2\pi i},\quad r_B=\frac{c_0\tau-d_0\eta_2}{2\pi i}.
\end{align*}
This, together with \eqref{eq2-16} in Remark \ref{Rmk2-4}, implies
\begin{equation}\label{lem-7.3} \tau_r=4\pi i(c_0-d_0\eta_1),\quad
\frac{\tau_s}{\tau_r}=-\frac{r_B}{s_B}=\frac{c_0\tau-d_0\eta_2}{c_0-d_0\eta_1}
=\tau+\frac{2\pi i}{\frac{c_0}{d_0}-\eta_1}.
\end{equation}
Recall again that $B$ can be a local coordinate of $\boldsymbol{a}$ as long as $\boldsymbol{a}$ is not a branch point.

\begin{Proposition} \label{Jacobian} Consider the vector-valued map $(E_{\tau}^\times)^n \to \Bbb R^2$ defined by
$$
\boldsymbol{a} \mapsto \phi(\boldsymbol{a}):= -4\pi \sum_{j = 1}^n \nabla G(a_j).
$$
Let $\Sigma_n(\tau)\ni B= u + iv \mapsto \boldsymbol{a}(B) \in Y_n(\tau)\setminus\{\text{branch points}\}$ be the local holomorphic parametrization, where $u,v\in\mathbb R$. Then the Jacobian $J(\phi\circ \boldsymbol{a})(u, v)$ is given by
\begin{align}
\det D(\phi\circ \boldsymbol{a})=\det \Big( \frac{\p\phi}{\p u}, \frac{\p\phi}{\p v}\Big) &= -\frac{4|c_0 - d_0\eta_1|^2}{\operatorname{Im}\tau}\operatorname{Im}\left(\tau+\frac{2\pi i}{\frac{c_0}{d_0}-\eta_1}\right)\\
&=-\frac{1}{4\pi^2\operatorname{Im}\tau}|\tau_r|^2\operatorname{Im}
\left(\frac{\tau_s}{\tau_r}\right)\nonumber,
\end{align}
where $c_0, d_0$ and $\tau_r, \tau_s$ are seen in \eqref{r1s2}-\eqref{lem-7.3}. Note from  \eqref{eq2-16} that $(\tau_r, \tau_s)\neq (0,0)$.
\end{Proposition}

\begin{proof} This proof is just by direct computations as in \cite{LW-CAG} where the case of $\boldsymbol{a}$ being branch points of $Y_n$ was considered (The difference is that near a branch point, $B$ can not be a local coordinate but $C$ is, where $C^2=\ell_n(B;\tau)$).

 Indeed,
denote $a_j = x_j +i y_j$ with $x_j, y_j\in\mathbb{R}$, $b=\operatorname{Im}\tau$ and $\phi = (\phi_1, \phi_2)^T$. Then
\[\sum_jx_j+i\sum_{j}y_j=\sum_ja_j=r+s\tau=r+s(\operatorname{Re}\tau+ib)\]
implies $s=\frac1b \sum_{j}y_j$.
 By (\ref{G_z}) and $\nabla G=(2\operatorname{Re}\frac{\partial G}{\partial z}, -2\operatorname{Im}\frac{\partial G}{\partial z})$, we have
\begin{equation} \label{grad-phi}
\begin{split}
\phi_1 &= 2\,{\rm Re}\,(\sum_j \zeta(a_j) - \eta_1 a_j),\\
\phi_2 &= -2\,{\rm Im}\,(\sum_j \zeta(a_j) - \eta_1 a_j) - \frac{4\pi}{b} \sum_j y_j.
\end{split}
\end{equation}
Recalling $(c_0, d_0)$ defined in \eqref{r1s2}, it follows from the chain rule, the Cauchy-Riemann equation and $\zeta'=-\wp$ that
{\allowdisplaybreaks
\begin{align*}
\p_u \phi_1 &= -2\,{\rm Re}\Big[\sum_j (\wp(a_j) + \eta_1)a_j'(B)\Big] = 2\,{\rm Re}\,(c_0 - d_0\eta_1),\\
\p_v \phi_1 &= 2\,{\rm Im}\Big[\sum_j (\wp(a_j) + \eta_1) a_j'(B)\Big] = -2\,{\rm Im}\,(c_0 - d_0\eta_1),\\
\p_u \phi_2 &= 2\,{\rm Im}\Big[\sum_j (\wp(a_j) + \eta_1) a_j'(B) \Big] - \frac{4\pi}{b} \sum_j \frac{\p y_j}{\p u} \\
&= -2\,{\rm Im}\,(c_0 - d_0\eta_1) - \frac{4\pi}{b}{\rm Im}\,d_0, \\
\p_v \phi_2 &= -2\,{\rm Re}\Big[\sum_j (\wp(a_j) + \eta_1) a_j'(B)\Big] - \frac{4\pi}{b} \sum_j \frac{\p y_j}{\p v} \\
&= -2\,{\rm Re}\,(c_0 - d_0\eta_1) - \frac{4\pi}{b} {\rm Re}\,d_0.
\end{align*}
}%
Hence the Jacobian is given by
\begin{equation*}
\begin{split}
&\quad\det \Big( \frac{\p\phi}{\p u}, \frac{\p\phi}{\p v}\Big)\\
 &=-4\left(|c_0 - d_0\eta_1|^2 + \frac{2\pi}{b} ({\rm Re}\,(c_0 - d_0\eta_1)\, {\rm Re}\,d_0 + {\rm Im}\,(c_0 - d_0\eta_1)\, {\rm Im}\,d_0)\right) \\
&= - 4\Big(|c_0 - d_0\eta_1|^2 + \frac{2\pi}{b} {\rm Re}\,\overline{d_0}(c_0 - s\eta_1)\Big)\\
&=-\frac{4|c_0 - d_0\eta_1|^2}{\operatorname{Im}\tau}\operatorname{Im}\bigg(\tau+\frac{2\pi i}{\frac{c_0}{d_0}-\eta_1}\bigg).
\end{split}
\end{equation*}
The proof is complete by applying \eqref{lem-7.3}.
\end{proof}

Now we are ready to prove Theorem \ref{DD-conj0}

\begin{theorem}[=Theorem \ref{DD-conj0}] \label{DD-conj}
Fix any $n \in \Bbb N^*$ and $\tau$. Suppose $\boldsymbol{p} = \{p_1, \cdots, p_n\} \in Y_n(\tau)$ is a nontrivial critical point of $G_n$, then there is a constant $c_{\boldsymbol{p}} >0 $ such that
\begin{equation}\label{eq-ex-non}
\det D^2 G_n(\boldsymbol{p}) = \frac{(-1)^n n^2}{4(2\pi)^{2n+2} \operatorname{Im}\tau} c_{\boldsymbol{p}} |\tau_r|^2\operatorname{Im}
\left(\frac{\tau_s}{\tau_r}\right).
\end{equation}
In particular, $\boldsymbol{p}$ is a degenerate critical point of $G_n$ if and only if
$|\tau_r|^2\operatorname{Im}
(\frac{\tau_s}{\tau_r})=0$, namely
\[\frac{\tau_s}{\tau_r}\in\mathbb{R}\cup\{\infty\}.\]
\end{theorem}

\begin{proof} First we consider the $n=1$, where $G_1(z)=-G(z)$. If $p$ is a nontrivial critical point of $G$, i.e. $p\neq -p$ in $E_{\tau}$ and so $\wp'(p)\notin \{0,\infty\}$, then $B=\wp(p)$ and a direct computation via (\ref{G_z}) leads to
\[\det D^2 G(p)=-\frac{|\eta_1+B|^2}{4\pi^2 \operatorname{Im}\tau}\operatorname{Im}\left(\tau-\frac{2\pi i}{\eta_1+B}\right).\]
Since it was computed in \cite[Section 5.1]{CL-AIM25} that
\[\tau_r=-4\pi i\frac{\eta_1+B}{\wp'(p)},\quad \frac{\tau_s}{\tau_r}=\tau-\frac{2\pi i}{\eta_1+B},\]
we obtain (\ref{eq-ex-non}) with
\[c_p=\frac{|\wp'(p)|^2}{64\pi^4\operatorname{Im}\tau}>0.\]

Now we consider the general case $n\geq 2$. Note that $D^2 G_n$ is a $2n \times 2n$ matrix and it seems impossible to compute $\det D^2 G_n(p)$ directly. Here we mainly follow the argument of \cite[Theorem 4.1]{LW-CAG} where Lin and Wang studied the Hessian of $G_n$ at trivial critical points $\boldsymbol a=-\boldsymbol a$. Note that different ideas are needed here due to the difference between nontrivial critical points and trivial critical points.

As pointed out in (\ref{ff3-1})-(\ref{ff4}), the system of equations (\ref{poho}) is equivalent to holomorphic equations  $g^1(\boldsymbol{a}) =\cdots = g^{n - 1}(\boldsymbol{a}) = 0$ with
\begin{equation} \label{zeta-eqn}
g^j(\boldsymbol{a}): = \sum_{k\ne j}^n (\zeta(a_j - a_k) + \zeta(a_k) - \zeta(a_j)),\quad 1\le j\le n-1,
\end{equation}
which defines the hyperelliptic curve $Y_n$ (see (\ref{hyper})), and the non-holomorphic equation $g^n(\boldsymbol{a}) = 0$ with
\begin{equation}\label{g-n}
g^n(\boldsymbol{a}) := \tfrac{1}{2}\phi(\boldsymbol{a}) = -2\pi \sum_{j = 1}^n \nabla G(a_j).
\end{equation}
Applying (\ref{G_z}), we easily obtain for $1\le j\le n-1$,
\begin{equation}\label{gia}
g^{j}(\boldsymbol{a})=-2\pi\left(\sum_{k\neq j}2G_z(a_j-a_k)-2nG_z(a_j)+\sum_{k=1}^n 2G_z(a_k)\right).
\end{equation}
Recall (\ref{poho}) that
$$
\nabla_j G_n(\boldsymbol{a})=\sum_{k\neq j}\nabla G(a_j-a_k)-n\nabla G(a_j),\quad \forall\,j.
$$
Then by taking into account that $\nabla G \mapsto 2G_z$ has matrix $\bigl(\begin{smallmatrix}1 & 0\\
0 & -1\end{smallmatrix}\bigr)$ and $$g^n = \frac{2\pi}{n}\sum_{j = 1}^n \nabla_{j} G_n,$$ we see that the equivalence between the map $g(\boldsymbol{a}) := (g^1(\boldsymbol{a}), \cdots, g^n(\boldsymbol{a}))^T$ and $-2\pi\nabla G_n(\boldsymbol{a})=-2\pi(\nabla_1 G_n(\boldsymbol{a}),\cdots,\nabla_n G_n(\boldsymbol{a}))^{T}$ is induced by a real $2n \times 2n$ matrix $A$ given by
\begin{equation*}
A= \begin{bmatrix}
1 & & & & & 1 &  \\
& 1 & & & & & -1 \\
& & \ddots & & & \vdots & \vdots\\
& & & 1& & 1 &  \\
& & & & 1& & -1 \\
& & & & & 1 &  \\
& & & & & & 1
\end{bmatrix}
\cdot
\begin{bmatrix}
1 & & & & &  &  \\
& -1 & & & & &  \\
& & \ddots & & & & \\
& & & 1& &  &  \\
& & & & -1& &  \\
\tfrac{-1}{n}& & \cdots& \tfrac{-1}{n}& & \tfrac{-1}{n} &  \\
&\tfrac{-1}{n} & \cdots& & \tfrac{-1}{n}& & \tfrac{-1}{n}
\end{bmatrix}.
\end{equation*}
In other words, by considering
$g^{k}=({\rm Re}g^{k}, {\rm Im}g^{k})^T$ for $1\le k\le n-1$ and $G_z=({\rm Re}G_z, {\rm Im}G_z)^T$, we have $2G_z=\bigl(\begin{smallmatrix}1 & 0\\
0 & -1\end{smallmatrix}\bigr)\nabla G$. Inserting this into (\ref{gia}), it is easy to obtain $$g(\boldsymbol{a})=-2\pi A \nabla G_n(\boldsymbol{a}).$$
Consequently,
\begin{equation} \label{J(g)}
J(g)(\boldsymbol{a}) = J(-2\pi A\nabla G_n)(\boldsymbol{a}) = \frac{(-1)^{n - 1}}{n^2} (2\pi)^{2n}\det D^2 G_n(\boldsymbol{a}),
\end{equation}
so it suffices for us to compute the real Jacobian $J(g)$.

Let $\boldsymbol{p} \in Y_n $ be a nontrivial critical point of $G_n$ and consider the holomorphic parametrization $B \mapsto \boldsymbol{a}(B)$ of $Y_n$ near $\boldsymbol{p}$, where $\boldsymbol{p}$ corresponds to $B = B_{\boldsymbol{p}}$.
Denote
$$
B= u + i v, \quad a_k = x^k + i y^k, \quad g^k = U^k + i V^k, \quad 1 \le k \le n.
$$
Along $Y_n$ (i.e. $g^1(\boldsymbol{a}) =\cdots = g^{n - 1}(\boldsymbol{a}) = 0$), it follows that (denote $g^j_k = \p g^j/\p a_k$)
\begin{equation} \label{der-level}
0 = \frac{\p g^j}{\p B} = \sum_{k = 1}^n g^j_k\, a_k'(B), \qquad 1 \le j \le n - 1.
\end{equation}

Note that $g^n = \tfrac{1}{2}\phi$ is not holomorphic (see (\ref{grad-phi}) for the additional linear term $-2\pi \sum_k y^k/b$), so we need to work with the real components $U^k, V^k$ and real variables $x^k, y^k$ and $u, v$ instead.
More precisely, note that (\ref{der-level}) takes the real form: For $1 \le j \le n - 1$,
\begin{equation} \label{der-level2}
\begin{split}
0 = \begin{bmatrix} U^j_u & U^j_v \\ V^j_u & V^j_v \end{bmatrix}
= \sum_{k = 1}^n
\begin{bmatrix} U^j_{x^k} & U^j_{y^k} \\ V^j_{x^k} & V^j_{y^k} \end{bmatrix} \begin{bmatrix} x^k_u & x^k_v \\ y^k_u & y^k_v \end{bmatrix}.
\end{split}
\end{equation}
The two rows are equivalent by the Cauchy--Riemann equation.

Since (\ref{b-a}) gives
\[1=(2n-1)\sum_j \wp'(a_j)a_j'(B),\]
and (\ref{eq-ex-aj}) says $\wp'(a_j)\neq 0,\infty$ for all $j$,
there exists $k$ such that
\begin{equation}\label{a-k'}
a_k'(B_{\boldsymbol{p}})\neq 0.
\end{equation}
By renaming $a_1,\cdots, a_n$ if necessarily, we may assume $k=n$, i.e. $a_n'(B_{\boldsymbol{p}})\neq 0$.
Then the elementary column operation on the $2n \times 2n$ real jacobian matrix $Dg$ is now replaced by the right multiplication with the matrix
\begin{equation*}
R_n := \begin{bmatrix}
1 & & & x^1_u & x^1_v \\
& 1 & & y^1_u & y^1_v \\
& & \ddots & \vdots & \vdots\\
& & & x^n_u & x^n_v \\
& & & y^n_u & y^n_v
\end{bmatrix},
\end{equation*}
because
$$
\begin{vmatrix} x^n_u & x^n_v \\ y^n_u & y^n_v \end{vmatrix} = |a'_n(B_{\boldsymbol{p}})|^2 \ne 0.
$$

Denote by $D'g$ the principal $2(n - 1) \times 2(n - 1)$ sub-matrix of $Dg$. Notice from (\ref{g-n}) that
$$
\tfrac{1}{2} D(\phi\circ \boldsymbol{a}) = \begin{bmatrix} U^n_u & U^n_v \\ V^n_u & V^n_v \end{bmatrix}
= \sum_{k = 1}^n
\begin{bmatrix} U^n_{x^k} & U^n_{y^k} \\ V^n_{x^k} & V^n_{y^k} \end{bmatrix} \begin{bmatrix} x^k_u & x^k_v \\ y^k_u & y^k_v \end{bmatrix},
$$
which is precisely the right bottom $2 \times 2$ sub-matrix of $(Dg)R_n$. Hence it follows from (\ref{der-level2}) that
$$
(Dg)R_n=\begin{bmatrix} D'g & 0 \\ * & \tfrac{1}{2} D(\phi\circ \boldsymbol{a}) \end{bmatrix},
$$
which can be used to calculate the determinant:
$$
\det Dg\, \det R_n = \det ((Dg)R_n) =\tfrac{1}{4} \det D' g\, \det D(\phi\circ \boldsymbol{a}).
$$
By (\ref{J(g)}), $\det Dg=J(g)(\boldsymbol{a})$ and Proposition \ref{Jacobian}, we get
\begin{equation*}
\det D^2 G_n(p) = \frac{(-1)^n n^2}{4 (2\pi)^{2n+2}{\rm Im}\,\tau} \frac{|\det D'^{\Bbb C} g(\boldsymbol{p})|^2}{|a_n'(B_{\boldsymbol{p}})|^2} |\tau_r|^2\operatorname{Im}\left(\frac{\tau_s}{\tau_r}\right),
\end{equation*}
where $D'^{\Bbb C} g$ denotes the principal $(n - 1) \times (n - 1)$ sub-matrix of the $n\times n$ complex-valued matrix $D^{\Bbb C}g$. Here we have used $\det D' g(\boldsymbol{p})=|\det D'^{\Bbb C} g(\boldsymbol{p})|^2$ because $g^k$ is holomorphic for any $1\le k\le n-1$.

It remains to prove that
\begin{equation} \label{cp-expression}
c_{\boldsymbol{p}} :=  \frac{|\det D'^{\Bbb C} g(\boldsymbol{p})|^2}{|a_n'(B_{\boldsymbol{p}})|^2} > 0.
\end{equation}
Denote $\tilde{g}:=(g^1,\cdots,g^{n-1})^T$.
Since $g^1 = 0, \cdots, g^{n - 1} = 0$ are the defining equations for $Y_n$ and $\boldsymbol{p} \in Y_n$ is a smooth point, then there is some $(n - 1) \times (n  - 1)$ sub-matrix of the $(n - 1) \times n$ matrix $D^{\mathbb{C}}\tilde{g}(\boldsymbol{a})$ such that its determinant is not zero at $\boldsymbol{p}$, i.e. the rank of $D^{\mathbb{C}}\tilde{g}(\boldsymbol{p})$ is $n-1$, so the dimension of the solution space of
\begin{equation}\label{eq-ex-linear}D^{\mathbb{C}}\tilde{g}(\boldsymbol{p})\cdot X=0,\quad X\in \mathbb{C}^n\end{equation}
is $1$. On the other hand, by letting $b_i$ equal to $(-1)^{1+i}$ times the determinant of the $(n-1)\times(n-1)$ sub-matrix of $D^{\mathbb{C}}\tilde{g}(\boldsymbol{p})$ without the $i$-th column, we have that
$(b_1,\cdots, b_n)^T$ $\neq (0,\cdots, 0)^T$ solves  (\ref{eq-ex-linear}). Since (\ref{der-level}) and \eqref{a-k'} say that $$(a_1'(B_{\boldsymbol{p}}), \cdots, a_n'(B_{\boldsymbol{p}}))^{T}\neq (0,\cdots, 0)^T$$ also solves (\ref{eq-ex-linear}), there exists a constant $c\neq 0$ such that
\[(b_1,\cdots, b_n)^T=c(a_1'(B_{\boldsymbol{p}}, \cdots, a_n'(B_{\boldsymbol{p}}))^{T}.\]
From here and $b_n=(-1)^{1+n}\det D'^{\Bbb C} g(\boldsymbol{p})$, we finally obtain \[c_{\boldsymbol{p}}=\frac{|b_n|^2}{|a_n'(B_{\boldsymbol{p}})|^2}=|c|^2>0.\]
The proof is complete.
\end{proof}

\begin{example} Let $n=2$ and suppose that $\boldsymbol{p}=(p_1,p_2)$ is a nontrivial critical point of $G_2$ with $B=3(\wp(p_1)+\wp(p_2))$. Since it was computed in \cite[Section 5.2]{CL-AIM25} that
\[\tau_r=-2\pi i\frac{B^2+3\eta_1B-\frac{3}{2}g_2}{\sqrt{\ell_2(B)}},\]
\[\frac{\tau_s}{\tau_r}=\tau-\frac{6\pi i B}{B^2+3\eta_1B-\frac{3}{2}g_2},\]
where
$$\ell_2(B)=(B^2-3g_2)\left(B^3-\frac94g_2B+\frac{27}{4}g_3\right)\neq 0,$$
we obtain
\[\det D^2 G_2(\boldsymbol{p}) = \frac{c_{\boldsymbol{p}} |B^2+3\eta_1B-\frac{3}{2}g_2|^2}{(2\pi)^{4} |\ell_2(B)|\operatorname{Im}\tau }\operatorname{Im}
\left(\tau-\frac{6\pi i B}{B^2+3\eta_1B-\frac{3}{2}g_2}\right).\]
In particular, $\boldsymbol{p}$ is a degenerate critical point of $G_2$ if and only if
\[\frac{\tau_s}{\tau_r}=\tau-\frac{6\pi i B}{B^2+3\eta_1B-\frac{3}{2}g_2}\in\mathbb{R}\cup\{\infty\}.\]
\end{example}

\section{Non-degeneray of critical points}

This final section is devoted to the proof of Theorem \ref{main-thm-1}. 
For any $\gamma=\bigl(\begin{smallmatrix}a & b\\
c & d\end{smallmatrix}\bigr)
\in SL(2,\mathbb{Z})$, we consider the Mobius transformation
$$\gamma(\tau):=\frac{a\tau+b}{c\tau+d}.$$
Recall that $\Gamma_{0}(2)$ is the congruence subgroup of $SL(2,\mathbb{Z})$ defined by
\[
\Gamma_{0}(2):=\left \{  \left.
\bigl(\begin{smallmatrix}a & b\\
c & d\end{smallmatrix}\bigr)
\in SL(2,\mathbb{Z})\right \vert c\equiv0\text{ }\operatorname{mod}2\right \},
\]
and $F_{0}$ is the basic fundamental domain of $\Gamma
_{0}(2)$ given by
\[
F_{0}:=\{ \tau \in \mathbb{H}\ |\ 0\leqslant  \text{Re}  \tau \leqslant
1\  \text{and}\ |\tau-\tfrac{1}{2}|\geqslant \tfrac{1}{2}\}.
\]
Then for any $\gamma\in \Gamma_0(2)$,
\[\gamma(F_0):=\{\gamma(\tau)\;:\; \tau\in F_0\}\]
is also a fundamental domain of $\Gamma
_{0}(2)$ and
\[\mathbb{H}=\bigcup_{\gamma\in \Gamma_0(2)}\gamma(F_0).\]

\begin{Remark}\label{Rmk4-1}
Fix any $\gamma=\bigl(\begin{smallmatrix}a & b\\
c & d\end{smallmatrix}\bigr)
\in \Gamma
_{0}(2)$, we denote $\tilde{\tau}=\gamma(\tau)=\frac{a\tau+b}{c\tau+d}$ for convenience. Since $ad-bc=1$ implies $$\mathbb Z+\mathbb{Z}\tau=\mathbb{Z}(c\tau+d)+\mathbb{Z}(a\tau+b)=(c\tau+d)\left(\mathbb Z+\mathbb{Z}\tilde{\tau}\right),$$
we know that $E_{\tau}$ and $E_{\tilde{\tau}}$ are conformally equivalent, and
$G(z;\tau)=G(\frac{z}{c\tau+d};\tilde{\tau})$, which implies 
$$G_n(\boldsymbol z;\tau)=G_n\Big(\frac{\boldsymbol z}{c\tau+d};\tilde{\tau}\Big).$$
Consequently, $\boldsymbol a$ is a (degenerate) critical point of $G_n(\cdot;\tau)$ if and only if $\frac{\boldsymbol a}{c\tau+d}$ is a (degenerate) critical point of $G_n(\cdot;\tilde{\tau})$. From here, we obtain
\begin{align*}
\mathcal O_n&=\{\tau\in \mathbb H\;:\; \text{$G_n(\boldsymbol{z};\tau)$ has degenerate critical points}\}\\
&=\bigcup_{\gamma\in \Gamma_0(2)}(\mathcal O_n\cap\gamma(F_0))
=\bigcup_{\gamma\in \Gamma_0(2)}\gamma(\mathcal O_n\cap F_0).
\end{align*}
Therefore, once we can prove that $\mathcal O_n\cap F_0$ is of Lebesgue measure zero, then so is $\gamma(\mathcal O_n\cap F_0)$ for any $\gamma
\in \Gamma
_{0}(2)$, and so is $\mathcal O_n$.
\end{Remark}

\begin{proof}[Proof of Theorem \ref{main-thm-1}] Define
\[\mathcal{O}_n^T:=\{\tau\in F_0\;:\; \text{$G_n(\boldsymbol{z};\tau)$ has degenerate trivial critical points}\},\]
\[\mathcal{O}_n^N:=\{\tau\in F_0\;:\; \text{$G_n(\boldsymbol{z};\tau)$ has degenerate nontrivial critical points}\},\]
then $\mathcal{O}_n\cap F_0=\mathcal{O}_n^T\cup \mathcal{O}_n^N$.

{\bf Step 1.} We prove that $\mathcal{O}_n^T$ is of Lebesgue measure zero.

Indeed, by applying results from \cite{Eremenko-GT}, it was proved in \cite{Lin-JDG25} that
\[\mathcal{O}_n^T=\big\{\tau\in F_0\,:\,\text{$\ell_n(\cdot;\tau)$ has multiple zeros}\big\}\bigcup \text{$\frac{3n(n+1)}{2}$ Lin-Wang curves},\]
where a Lin-Wang curve was first defined by \cite{Eremenko-GT}, and is a real analytic curve in $F_0$ without self-intersection and can be parametrized by one of $(-\infty, 0), (0,1), (1,+\infty)$. This implies that the union of $\frac{3n(n+1)}{2}$ Lin-Wang curves is of of Lebesgue measure zero.
Furthermore, since the discriminant of $\ell_n(\cdot;\tau)$ is a non-zero modular form, it is known (cf. \cite{CLW}) that the set $\{\tau\in F_0:\text{$\ell_n(\cdot;\tau)$ has multiple zeros}\}$ is finite. Therefore, $\mathcal{O}_n^T$ is of Lebesgue measure zero.

{\bf Step 2.} We prove that $\mathcal{O}_n^N$ is of Lebesgue measure zero.

Let $\tau\in F_0$ and
suppose $\pm\boldsymbol a=\pm(a_1,\cdots, a_n)$ is a pair of nontrivial critical points of $G_{n}(\boldsymbol z;\tau)$. By choosing suitable representative $a_j$'s from $a_j+\mathbb Z+\mathbb Z\tau$ and by replacing $\boldsymbol a$ by $-\boldsymbol a$ if necessary, we may assume
$\sum_{j}a_j=r+s\tau$ with $(r,s)\in [0,1]\times [0,\frac12]$. Then Theorem \ref{thm-clw2} implies $(r,s)\in[0,1]\times [0,\frac12]\setminus\frac{1}{2}\mathbb{Z}^2$ and $Z_{r,s}^{(n)}(\tau)=0$.
Therefore, it follows from the definition \eqref{eq-en} of $\mathcal{E}_n$ that
\begin{align}\label{enen}
\mathcal{E}_n\cap F_0&=\{\tau\in F_0\,:\, G_n(\boldsymbol{z};\tau)\;\text{has nontrivial critical points}\}\\
&=\{\tau\in F_0\,:\, Z_{r,s}^{(n)}(\tau)=0\;\text{for some }\,(r,s)\in[0,1]\times [0,\tfrac12]\setminus\tfrac{1}{2}\mathbb{Z}^2\}.\nonumber
\end{align}

Conversely, we
define four open triangles in $\mathbb R^2$:
{\allowdisplaybreaks
\begin{align}\label{rectangle}
&\triangle_{0}:=\{(r,s)\mid0<r,s<\tfrac{1}{2},\text{ }r+s>\tfrac{1}{2}\},\nonumber\\
&\triangle_{1}:=\{(r,s)\mid \tfrac{1}{2}<r<1,\text{ }0<s<\tfrac{1}{2},\text{
}r+s>1\},\nonumber\\
&\triangle_{2}:=\{(r,s)\mid \tfrac{1}{2}<r<1,\text{ }0<s<\tfrac{1}{2},\text{
}r+s<1\},\nonumber\\
&\triangle_{3}:=\{(r,s)\mid r>0,\text{ }s>0,\text{ }r+s<\tfrac{1}{2}\}.\nonumber
\end{align}
}%
Then $[0,1]\times \lbrack0,\frac{1}{2}]=\cup_{k=0}^{3}\overline
{\triangle_{k}}$. 
It was proved in \cite{Lin-JDG25} that there are exactly $\frac{n(n+1)}{2}$ analytic maps
\[\tau_j: \Omega_j\to F_0,\quad 1\leq j\leq \frac{n(n+1)}{2},\]
with
\[\Omega_j\in \{\triangle_{0}, \triangle_{1}, \triangle_{2}, \triangle_{3}\},\quad \forall j,\]
such that the followings hold:
\begin{itemize}
\item[(a)]
$Z_{r,s}^{(n)}(\tau_j(r,s))\equiv0$ for any $j$. Conversely, if $\tau\in F_0$ and $(r,s)\in [0,1]\times [0,\frac12]\setminus\frac{1}{2}\mathbb{Z}^2$ satisfy $Z_{r,s}^{(n)}(\tau)=0$, then there is a unique $j$ such that $(r,s)\in \Omega_j$ and $\tau=\tau_j(r,s)$.

\item[(b)] Recalling \eqref{enen}, there holds
\[\mathcal{E}_n\cap F_0=\{\tau\in F_0\,:\, G_n(\boldsymbol{z};\tau)\;\text{has nontrivial critical points}\}
= \bigcup_{j=1}^{\frac{n(n+1)}{2}}\tau_j(\Omega_j).\]
\item[(c)] For $\tau\in \mathcal{E}_n\cap F_0$ and
 $\boldsymbol a=(a_1,\cdots, a_n)$ being a nontrivial critical point of $G_{n}(\boldsymbol z;\tau)$ such that $\sum_{j}a_j=r+s\tau$ with $(r,s)\in[0,1]\times [0,\frac12]\setminus\frac{1}{2}\mathbb{Z}^2$, there is a unique $j$ such that $(r,s)\in \Omega_j$ and $\tau=\tau_j(r,s)$.
\end{itemize}

By (c), it follows from Theorem \ref{thm-sz} and the implicit function theorem that this map $\tau_j: \Omega_j\to F_0$ coincides with the local map in  Remark \ref{Rmk2-4} and Corollary \ref{Coro-11}, at least in a small neighborhood of $(r,s)$ in $\mathbb R^2$. 
Then Theorem \ref{DD-conj0} implies
\[\det D^2 G_n(\boldsymbol{a};\tau) = \frac{(-1)^n n^2}{4(2\pi)^{2n+2} \operatorname{Im}\tau} c_{\boldsymbol{a}} |\tau_{j,r}(r, s)|^2\operatorname{Im}
\left(\frac{\tau_{j,s}}{\tau_{j,r}}(r, s)\right),\]
with $c_{\boldsymbol a}>0$.
Note that $$\frac{\tau_{j,s}}{\tau_{j,r}}\in\mathbb{R}\cup\{\infty\}\quad\Longleftrightarrow\quad  \det\begin{pmatrix}\frac{\partial \operatorname{Re}\tau_j}{\partial r} & \frac{\partial \operatorname{Re}\tau_j}{\partial s}\\\frac{\partial \operatorname{Im}\tau_j}{\partial r} & \frac{\partial \operatorname{Im}\tau_j}{\partial s}\end{pmatrix}=0.$$
So by applying Sard theorem to $(\operatorname{Re}\tau_j, \operatorname{Im}\tau_j): \Omega_j\subset\mathbb{R}^2\to \mathbb{R}^2$, it follows that
$$\tau_j\left(\Big\{(r,s)\in\Omega_j : \frac{\tau_{j,s}}{\tau_{j,r}}\in\mathbb{R}\cup\{\infty\}\Big\}\right)$$
is of Lebesgue measure zero. Therefore, by Theorem \ref{DD-conj0}, we finally arrive that
\begin{align*}
\mathcal{O}_n^N&=\{\tau\in F_0\;:\; \text{$G_n(\boldsymbol{z};\tau)$ has degenerate nontrivial critical points}\}\\
&=\{\tau\in \mathcal{E}_n\cap F_0\;:\; \text{$G_n(\boldsymbol{z};\tau)$ has degenerate nontrivial critical points}\}\\
&=\bigcup_{j=1}^{\frac{n(n+1)}{2}}\tau_j\left(\Big\{(r,s)\in\Omega_j : \frac{\tau_{j,s}}{\tau_{j,r}}\in\mathbb{R}\cup\{\infty\}\Big\}\right)
\end{align*}
is also of Lebesgue measure zero.

By Step 1 and Step 2, we see that $\mathcal{O}_n\cap F_0=\mathcal{O}_n^T\cup \mathcal{O}_n^N$ is of Lebesgue measure zero. Then by Remark \ref{Rmk4-1}, we conclude that $\mathcal{O}_n$ is of Lebesgue measure zero. The proof is complete.
\end{proof}

\subsection*{Acknowledgements} Z. Chen was supported by National Key R\&D Program of China (Grant 2023YFA1010002) and NSFC (No. 12222109). E. Fu was supported by NSFC (No. 12401188) and BIMSA Start-up Research Fund.

\end{document}